\documentclass{kms-b}

\newcommand{\publname}{Bull.~Korean Math.~Soc.}
\issueinfo{}
  {}
  {}
  {}
\pagespan{1}{}
\copyrightinfo{}
  {Korean Mathematical Society}

\allowdisplaybreaks

\theoremstyle{plain}
\newtheorem{theorem}{Theorem}[section]

\newtheorem{corollary}[theorem]{Corollary}
\theoremstyle{definition}
\theoremstyle{remark}
\newtheorem{remark}[theorem]{Remark}

\usepackage{mathrsfs}
\usepackage{amsthm}
\usepackage{graphicx, amscd, amsmath, amsfonts, amssymb}

\begin{document}

\title[Mathieu-type series of $(p,q)$--extended Gaussian ${}_2F_1$]
{Mathieu-type series built by $(p, q)$--extended Gaussian hypergeometric function}

\author[J. Choi]{Junesang Choi}
\address{Junesang Choi \\ Department of Mathematics \\ Dongguk University \\ Gyeongju 38066, Republic of Korea}
\email{junesang@dongguk.ac.kr}

\author[R. K. Parmar]{Rakesh K. Parmar}
\address{Rakesh K. Parmar \\ Department of Mathematics \\ Government College of Engineering and Technology \\ Bikaner-334004, Rajasthan State, India} \email{rakeshparmar27@gmail.com}

\author[T. K. Pog\'any]{Tibor K. Pog\'any}
\address{Tibor K. Pog\'any \\ Faculty of Maritime Studies \\ University of Rijeka \\ 51000 Rijeka, Croatia \\ $and$ \\
Applied Mathematics Institute \\ \'Obuda University \\ 1034 Budapest, Hungary} 
\email{poganj@pfri.hr}


\begin{abstract}
The main purpose of this paper is to present closed integral form expressions for the Mathieu-type $\boldsymbol a$-series and its 
associated alternating version whose terms contain a $(p, q)$--extended Gauss' hypergeometric function. 
Certain upper bounds for the two series are also given.
\end{abstract}

\maketitle

\allowdisplaybreaks

\section{\bf Introduction and preliminaries}
In the recent articles Pog\'any, either alone and/or with his co--workers Baricz, Butzer, Saxena, Srivastava and Tomovski
\cite{baricz0}, \cite{Pog-1, Pog-2, Pog-3,  Pog-Sax, Pog-4, Pog-5, Pog-6} considered special general Mathieu--type series and
their alternating variants whose terms contain the various special functions, for example, Gauss
hypergeometric function ${}_2F_1$, generalized hypergeometric ${}_pF_q$, Meijer $G$-functions and so on. The derived results
concern, among others, closed integral form expressions for the considered series and bilateral bounding inequalities. Here
we are interested in giving integral expressions for the Mathieu--type series and its alternating variants  built by terms which contain the $(p, q)$--extended Gauss' hypergeometric function which generalizes the so--called
$p$--extension of the $p$--Gaussian hypergeometric function \cite{Ch-Qa-Sr-Pa, Parmar-Pogany}.
These functions are built by first changing some terms in the defining series into Beta-function and then
replacing the Beta-function with its  $p$--variant and $(p, q)$--variant.
The above mentioned extensions, generalizations and unifications of Euler's Beta function together with a set of related higher transcendental
hypergeometric type special functions have been investigated recently by several authors, for instance,
 one may refer to \cite{Ch-Qa-Ra-Zu, Ch-Qa-Sr-Pa, Ch-Zu-02-Book}.
 In particular, Chaudhry {\it et al.} \cite[p. 20, Eq. (1.7)]{Ch-Qa-Ra-Zu} introduced the $p$--extension of the
Eulerian Beta function ${\rm B}(x, y)$:
   \[ {\rm B}(x, y; p) = \int_0^1\, t^{x-1}\,(1-t)^{y-1}\, {\rm e}^{-\frac{p}{t(1 - t)}}  {\rm d}t \qquad (\Re (p)>0), \]
whose special case when $p=0$ reduces to the familiar Beta function ${\rm B}(x, y)$ $(\min\{ \Re (x),\,\Re (y)\}>0)$
(see, \emph{e.g.}, \cite[Section 1.1]{Sr-Ch-12}).
They extended Macdonald (or modified Bessel function of the second kind), error and Whittaker functions by using
the ${\rm B}(x, y; p)$.
Also Chaudhry {\it et al.} \cite{Ch-Qa-Sr-Pa} used
the ${\rm B}(x,y; p)$ to extend  Gaussian hypergeometric and  confluent (Kummer's) hypergeometric functions in the following manner:
   \begin{equation} \label{A2}
      F_p(a, b;c; z)  =  \sum_{n\geq 0}^\infty\,(a)_n\, \frac{{\rm B}(b+n,\,c-b\,;\,p)}{{\rm B}(b,\,c-b)}\,\frac{z^n}{n!},                            
	 \end{equation}
where $p \geq 0;\, \Re(c)>\Re(b)>0\,;|z|<1$ and
   \begin{equation} \label{A3}
      \Phi_p(b;c; z) = \sum_{n \geq 0} \frac{{\rm B}(b+n,\,c-b\,;\,p)}{{\rm B}(b,\,c-b)}\,\frac{z^n}{n!},			                          
   \end{equation}
where $p \geq 0 \,;\, \Re(c)>\Re(b)>0$, respectively. It is noted that the special case of \eqref{A2} and \eqref{A3} when $p=0$
yield,  respectively, the Gaussian hypergeometric function ${}_2F_1(a,b;c;z)$ and the confluent (Kummer's) hypergeometric function
${}_1F_1(b;c;z)$ (see, \emph{e.g.}, \cite[Section 1.5]{Sr-Ch-12}).

Recently, Choi {\it et al.} \cite{Ch-Ra-Pa} have introduced {\it further} extensions of ${\rm B}(x, y; p)$, $p$--extended Gauss' hypergeometric
series $F_p(a, b;c; z)$ and {\it a fortiori} the $p$--Kummer (or confluent hypergeometric) $\Phi_p(b;c; z)$ as follows:
   \[ {\rm B}(x,y;p,q) =  \int_0^1 t^{x-1}(1-t)^{y-1} \, {\rm e}^{-\frac{p}{t} - \frac{q}{1 - t}}  {\rm d}t
      \qquad (\min\{\Re(p),\, \Re(q)\} \geq 0). \]
Here  $\min\{\Re(x), \Re(y)\}>0$ if $p=0=q$. Next, for all $\Re(c)>\Re(b)>0$  
   \[ F_{p,q}(a, b;c; z)  =  \sum_{n \geq 0}\,(a)_n\, \frac{{\rm B}(b+n,\,c-b\,;\,p,q)}{{\rm B}(b,\,c-b)}\,\frac{z^n}{n!}	\qquad (|z|<1), \] 
and 
   \begin{equation}\label{A6}
      \Phi_{p,q}(b;c; z)  =  \sum_{n \geq 0}\, \frac{{\rm B}(b+n,\,c-b\,;\,p,q)}{{\rm B}(b,\,c-b)}\,\frac{z^n}{n!}
			                            \qquad (\Re(c)>\Re(b)>0).
   \end{equation}
The $F_{p,q}(a, b;c; z)$ and the $\Phi_{p,q}(b;c; z)$ are called  {\it $(p, q)$--extended Gauss'} and {\it $(p, q)$--extended Kummer 
hypergeometric functions}, respectively. For their related properties, integral representations, differentiation formulas, Mellin 
transform, recurrence relations and certain summations, the interested reader may refer to \cite{Ch-Ra-Pa}. For a $(p, q)$--extended 
Srivastava's triple generalized $H_{p, q, A}$ function, one may see \cite{Parmar-Pogany}.

Now,  by imposing the $F_{p,q}(a, b;c; z)$ input--kernel instead of the originally used ${}_2F_1$ in the summands of the Mathieu--type series 
in \cite{Pog-1}, we extend to define the Mathieu--type $\mathbf a$--series $\mathfrak{F}_{\lambda,\eta}$ and its alternating variant 
$\mathfrak{\widetilde{F}}_{\lambda,\eta}$ in the form of series 
   \begin{equation} \label{A7}
      \mathfrak{F}_{\lambda,\eta}(F_{p,q};\boldsymbol{a};r) := \sum_{n \geq 1} \frac{F_{p,q}\,(\lambda,\, b;\,c;\, -\frac{r^{2}}{a_{n}})}
			             {a_{n}^{\lambda}(a_{n}+r^{2})^{\eta}}\qquad \left(p \geq 0,\;q \geq 0;\,\lambda,\eta,r \in \mathbb{R}^{+}\right)
   \end{equation}
and in the same range of parameters 
   \begin{equation} \label{A8}
      \mathfrak{\widetilde{F}}_{\lambda,\eta}(F_{p,q};\boldsymbol{a};r) := \sum_{n \geq 1} \frac{(-1)^{n-1}
			                 F_{p,q}\,(\lambda,\, b;\,c;\, -\frac{r^{2}}{a_{n}})}{a_{n}^{\lambda}(a_{n}+r^{2})^{\eta}}.
   \end{equation}
Here and in what follows, let $\mathbb{R}$ and $\mathbb{R}^{+}$ be the sets of real and  positive real numbers, respectively.

The main purpose of this note is to present integral representations and allied bounding inequalities for these functions in
the widest range of the parameters involved.

\section{Integral representations of $\mathfrak{F}_{\lambda,\eta}(F_{p,q};\boldsymbol{a};r)$ and
$\mathfrak{\widetilde{F}}_{\lambda,\eta}(F_{p,q};\boldsymbol{a};r)$}

In this section, we first give closed integral form expressions for the series
$\mathfrak{F}_{\lambda,\eta}(F_{p,q};\boldsymbol{a};r)$ and $\mathfrak{\widetilde{F}}_{\lambda,\eta}(F_{p,q};\boldsymbol{a};r)$.
Then we give some special cases of our first main result.

\begin{theorem}\label{thm-1}
 Let $\lambda$,\;$\eta$,\;$r \in \mathbb{R}^+$ and let $\boldsymbol{a} = (a_n)_{n \geq 1}$ be  a real sequence
 which increases monotonically and tends to $\infty$. Then for $\min\{\Re(p),\Re(q)\} \geq 0$ we have
    \begin{equation}\label{C2}
       \mathfrak{F}_{\lambda,\eta}(F_{p,q};\boldsymbol{a};r)      = \lambda \;\mathscr I_{p, q}(\lambda+1,\eta)
		                  + \eta\; \mathscr I_{p, q}(\lambda,\eta+1)
    \end{equation}
  and
  \begin{equation}\label{C21}
   \mathfrak{\widetilde{F}}_{\lambda,\eta}(F_{p,q};\boldsymbol{a};r)
     = \lambda \;\mathscr{\widetilde{I}}_{p, q}(\lambda+1,\eta)
	                  + \eta\; \mathscr{\widetilde{I}}_{p, q}(\lambda,\eta+1),
  \end{equation}
  where
  \begin{equation}\label{C3}
  \mathscr I_{p, q}(\lambda, \eta)   = \int_{a_{1}}^{\infty} \frac{F_{p,q}\,(\lambda,\, b;\,c;\,
	          - \frac{r^{2}}{x})[a^{-1}(x)]}{x^{\lambda}(x+r^{2})^{\eta}}\;{\rm d}x
  \end{equation}
  and
  \begin{equation}\label{C31}
  \mathscr{\widetilde{I}}_{p, q}(\lambda,\eta) = \int_{a_{1}}^{\infty} \frac{F_{p,q}\,(\lambda,\, b;\,c;\,
			                                   - \frac{r^{2}}{x})\;\sin^{2}\left(\frac{\pi}{2}[a^{-1}(x)]\right)}
	 {x^{\lambda}(x+r^{2})^{\eta}}\;{\rm d}x
  \end{equation}
  and $a:\mathbb{R}^{+}\mapsto\mathbb{R}^{+}$ is an increasing function such that $a(x)|_{x \in \mathbb{N}}=\boldsymbol{a}$,
$a^{-1}(x)$ denotes the inverse of $a(x)$ and $[a^{-1}(x)]$ stands for the integer part of the quantity $a^{-1}(x)$.
\end{theorem}

\begin{proof}
Consider the Laplace transform formula of the extended Kummer's function $ t^{\lambda-1}\;\Phi_{p,q}(b;c; z)$.
By using the definition \eqref{A6}, for real $\omega$,  it follows easily
   \begin{equation}\label{Laplace}
      F_{p,q}\left(\lambda, b;c; \frac{\omega}{z}\right) = \frac{z^{\lambda}}{\Gamma(\lambda)}\int_{0}^{\infty}
			            {\rm e}^{-zt} t^{\lambda-1}\;\Phi_{p,q}(b;c;\omega t )\, {\rm d}t.
   \end{equation}
Taking $\xi = a_{n} + r^{2}$ in the familiar Gamma formula:
   \[ \Gamma(\eta)\xi^{-\eta}=\int_{0}^{\infty} e^{-\xi t}  t^{\eta-1}\, {\rm d}t \qquad
      (\min\{\Re(\xi),\,\Re(\eta)\}>0) \]
and after rearrangement by specifying $\omega=-r^{2}$, $z=a_{n}$, in \eqref{Laplace}, the function
$\mathfrak{F}_{\lambda,\eta}(F_{p,q};\boldsymbol{a};r)$ becomes
   \[ \mathscr{I}_{p,q}(\lambda,\eta) = \int_0^{\infty}\int_0^{\infty} \frac{{\rm e}^{-r^2 s} \,t^{\lambda-1}s^{\eta-1}}
                                  {\Gamma(\lambda)\Gamma(\eta)} \left(\sum_{n \geq 1} {\rm e}^{-a_n(t+s)}\right)
																	\Phi_{p,q}(b;c;-r^{2}t)\;{\rm d}t\;{\rm d}s\,. \]
Using the Cahen formula \cite{Cah} for summing up the Dirichlet series in the technique developed in \cite{Pog-5}, we conclude
   \[ \mathcal{D}_{a}(t+s) = \sum_{n \geq 1} {\rm e}^{-a_n(s+t)}
	                         = (s+t)\int_{a_{1}}^{\infty} e^{-(t+s)}[a^{-1}(x)]\;{\rm d} x\,.\]
This gives
   \begin{align}\label{C5}
      \mathscr{I}_{p,q}(\lambda,\eta) &= \frac1{\Gamma(\lambda)\Gamma(\eta)} \int_{0}^{\infty} \int_{0}^{\infty}\int_{a_{1}}^{\infty}
			                             {\rm e}^{-(r^{2}+x) s-tx}(t+s)t^{\lambda-1}s^{\eta-1}[a^{-1}(x)]\notag \\
                                &\qquad \times\; \Phi_{p,q}(b;c;-r^{2}t)\;{\rm d}t\;{\rm d}s\;{\rm d}x
																 =: \mathcal{I}_{t}+\mathcal{I}_{s}\,,
   \end{align}
where
   \begin{align} \label{C6}
      \mathcal{I}_{t} &= \frac1{\Gamma(\eta)}\int_{0}^{\infty}\left(\int_{a_1}^{\infty}\left(\int_0^{\infty}
			                   \frac{{\rm e}^{-xt}t^{\lambda}}{\Gamma(\lambda)}\Phi_{p,q}(b;c;-r^{2}t)\;{\rm d}t\right)
												 {\rm e}^{-xs}[a^{-1}(x)]\;{\rm d}x\right)\notag \\
											&\qquad	 \times {\rm e}^{-r^{2}s}\, s^{\eta-1}\, {\rm d}s \notag \\
                      &= \lambda \int_{a_{1}}^{\infty}\left(\int_{0}^{\infty}
											   \frac{s^{\eta-1}}{\Gamma(\eta)}{\rm e}^{-(x+r^{2})s}\;{\rm d}s\right) \frac{[a^{-1}(x)]}{x^{\lambda+1}}\;
												 F_{p,q}\,\left(\lambda+1,\, b;\,c;\, -\frac{r^{2}}{x}\right)\;{\rm d}x\notag \\
                      &= \lambda\;\int_{a_{1}}^{\infty}\frac{[a^{-1}(x)]}{x^{\lambda+1}(x+r^{2})^{\eta}}\;
											   F_{p,q}\,\left(\lambda+1,\, b;\,c;\, -\frac{r^{2}}{x}\right)\;{\rm d}x
											 = \lambda \;\mathscr{I}(\lambda+1,\eta)\,.
   \end{align}
Similarly we get
   \begin{align} \label{C7}
      \mathcal{I}_{s} &= \eta\;\int_{a_{1}}^{\infty}\frac{[a^{-1}(x)]}{(x+r^{2})^{\eta+1}}\left(\int_{0}^{\infty}
                         \frac{{\rm e}^{-xt}t^{\lambda-1}}{\Gamma(\lambda)}\Phi_{p,q}(b;c;-r^{2}t)\;{\rm d}t\right)\;{\rm d}x\notag \\
                      &= \eta\;\int_{a_{1}}^{\infty}\frac{[a^{-1}(x)]}{x^{\lambda}(x+r^{2})^{\eta+1}}\;
											   F_{p,q}\,\left(\lambda,\, b;\,c;\, -\frac{r^{2}}{x}\right)\;{\rm d}x = \eta \;\mathscr{I}(\lambda,\eta+1).
   \end{align}
Now, applying \eqref{C6} and \eqref{C7} to \eqref{C5} we deduce the expression \eqref{C2}.

The derivation of \eqref{C21} is done with a similar procedure as in getting \eqref{C2}.
As to the alternating Dirichlet series $\mathcal{D}_{a}(x)$
integral form, having in mind again the Cahen formula, we have \cite{Pog-5}
   \[ \mathcal{\widetilde{D}}_{a}(x) = \sum_{n \geq 1}(-1)^{n-1}{\rm e}^{-a_{n}(x)}
	                                   = x\;\int_{a_{1}}^{\infty} {\rm e}^{-xt}\widetilde{A}(t)\;{\rm d}t,\]
and therefore
   \[ \mathcal{\widetilde{D}}_{a}(x) = x\;\int_{a_{1}}^{\infty} {\rm e}^{-xt}\sin^{2}\left(\frac{\pi}{2}[a^{-1}(x)]\right)\;{\rm d}t,\]
since the counting function turns out to be
   \[ \widetilde{A}(t) = \sum_{n\colon a_{n}\leq t}(-1)^{n-1} = \frac{1-(-1)^{[a^{-1}(t)]}}{2}
	                            =\sin^{2}\left(\frac{\pi}{2}[a^{-1}(t)]\right)\, .\]
Hence, because
   \[ \mathcal{\widetilde{D}}_{a}(t+s) = (t+s)\;\int_{a_{1}}^{\infty} {\rm e}^{-(t+s)x}
	                                       \sin^{2}\left(\frac{\pi}{2}[a^{-1}(t)]\right)\;{\rm d}x,\]
we conclude \eqref{C21} by carrying out the obvious remaining steps.
\end{proof}

Now, in the case $p=q$, Theorem \ref{thm-1} reduces to the following corollary.

\begin{corollary}\label{cor-1}
Let $\lambda$,\;$\eta$,\;$r \in \mathbb{R}^+$ and let $\boldsymbol{a} = (a_n)_{n \geq 1}$ be  a real sequence
 which increases monotonically and tends to $\infty$.
  Then for $\Re(p)\geq 0$ we have
  \[ \mathfrak{F}_{\lambda,\eta}(F_{p};\boldsymbol{a};r)        = \lambda \;\mathscr{J}_p(\lambda+1,\eta)
	                         +\eta\; \mathscr{J}_p(\lambda,\eta+1),\]
  and
  \[ \mathfrak{\widetilde{F}}_{\lambda,\eta}(F_{p};\boldsymbol{a};r) = \lambda \;\mathscr{\widetilde{J}}_p(\lambda+1,\eta)
	      +\eta\; \mathscr{\widetilde{J}}_p(\lambda,\eta+1), \]
  where
  \[ \mathscr{J}_p(\lambda,\eta)         = \int_{a_1}^{\infty} \frac{F_{p}\,(\lambda,\, b;\,c;\,
	                 - \frac{r^{2}}{x})[a^{-1}(x)]}{x^{\lambda}(x+r^{2})^{\eta}}\;dt,\]
  and
   \[ \mathscr{\widetilde{J}}_p(\lambda,\eta) = \int_{a_1}^{\infty} \frac{F_{p}\,(\lambda,\, b;\,c;\,
			                                   - \frac{r^{2}}{x})\;\sin^{2}\left(\frac{\pi}{2}[a^{-1}(x)]\right)}
		   {x^{\lambda}(x+r^{2})^{\eta}}\;{\rm d}t. \]
  \end{corollary}

\vskip 3mm

\begin{remark}\label{rmk-1}
The special case of {\rm Theorem \ref{thm-1}} when $p=q=0$ is seen to immediately reduce
to the Gauss hypergeometric function ${}_2F_1$ result  in \cite{Pog-1}.\end{remark}

\section{Bounding inequalities for the $(p,q)$--extended Mathieu--type series} 

Very recently Parmar and Pog\'any \cite{Parmar-Pogany}  have established an upper bound for the $(p, q)$--extended Beta function 
${\rm B}(x, y; p,q)$ (see \cite[Lemma 2]{Parmar-Pogany}). Namely, we have	
	 \[ {\rm B}(x,y;p,q) \leq {\rm e}^{ -(\sqrt{p}+\sqrt{q})^2}\,{\rm B}(x,y) \qquad (\min\{x, y, p, q\}\geq 0), \]
by observing 
   \[ \sup_{0 < t< 1} {\rm e}^{- \frac{p}{t} - \frac{q}{1-t}} = {\rm e}^{ -(\sqrt{p}+\sqrt{q})^2} =: \mathfrak E_{p, q}
			                \qquad (\min\{p,\, q\} \geq 0). \]
Here we recall  the following results in \cite[Theorem 8, Eqs. (3.2) and (3.3)]{Parmar-Pogany}:
  \begin{equation}\label{G3}
    \left| F_{p,q}(a, b;c; z)\right|  \leq {\rm e}^{ -(\sqrt{p}+\sqrt{q})^2}\, {}_2F_1(a, b; c; |z|)
  \end{equation}
 and 
    \[ \left|\Phi_{p,q}(b;c; z)\right|\leq {\rm e}^{ -(\sqrt{p}+\sqrt{q})^2}\, \Phi(b; c; |z|),\]
where $\min\{p,\,q\}\geq 0$,  $c>b>0$ and $|z|<1$. Also we need to recall  a certain Luke's upper bound for the Gaussian
hypergeometric function (see \cite[p. 52, Eq. (4.7)]{Luke}): For all $b \in(0, 1],\, c \geq a>0$ and $z>0$, 
the following inequality holds true:
   \begin{align} \label{G5}
	    {}_2F_1(a, b; c; -z) < 1 - \frac{2ab(c+1)}{c(a+1)(b+1)} \Big[1- \frac{2(c+1)}{2(c+1)+(a+1)(b+1)\,z} \Big] \, .
	 \end{align}
For simplicity, the following notation is introduced:
  \begin{equation}\label{u-notation}
   \mathscr U_a(\lambda, \eta) := \int_{a_1}^\infty \frac{[a^{-1}(x)]}{x^{\lambda}(x+r^{2})^{\eta}}\;{\rm d}x.
  \end{equation}
In the sequel we consider  Mathieu--type series \eqref{A7} and \eqref{A8} in which the defining functions
$a \colon \mathbb R_+ \mapsto \mathbb R_+$ behave so that $\mathscr U_a(\lambda, \eta)$ converges.

\vskip 3mm

\begin{theorem}\label{thm-2}
 Let $\lambda \in (0, 1]$ and $ \eta \in \mathbb{R}^+$ and let $\boldsymbol{a} = (a_n)_{n \geq 1}$ be  a real sequence
 which increases monotonically and tends to $\infty$. Then, for all $r \in (0, \sqrt{a_1}\;)$,   $\min\{p,\,q\} \geq 0$ and $c>b>0$, we have
   \begin{align} \label{G6}
	    \mathfrak F_{\lambda,\eta}(F_{p,q};\boldsymbol{a};r) &\leq \lambda \;\mathfrak E_{p, q}
			    \Bigg\{ \left(1-\frac{2(\lambda+1) b(c+1)}{c(\lambda+2)(b+1)}\right) \mathscr U_a(\lambda+1, \eta)  \nonumber\\
			 &\qquad + \frac{4(\lambda+1) b(c+1)^2\,\mathscr U_a(\lambda, \eta)}
											{c(\lambda+2)(b+1)\,[(\lambda+2)(b+1)r^2+2(c+1)a_1]} \Bigg\} \nonumber \\
			 &\qquad + \eta\;\mathfrak E_{p, q} \Bigg\{ \left(1-\frac{2\lambda b(c+1)}{c(\lambda+1)(b+1)}\right)
								      \mathscr U_a(\lambda, \eta+1)  \nonumber \\
			 &\qquad + \frac{4\lambda b(c+1)^2\,\mathscr U_a(\lambda-1, \eta+1)}
											{c(\lambda+1)(b+1)\,[(\lambda+1)(b+1)r^2+2(c+1)a_1]} \Bigg\}\,.
	 \end{align}
Moreover, for all $\lambda+\eta>1$, $r \in (0, \sqrt{a_1}\;)$,   $\min\{p,\, q\} \geq 0$ and $c>b>0$, we have
   \begin{align} \label{G7}
	    \mathfrak{\widetilde{F}}_{\lambda,\eta}(F_{p,q};\boldsymbol{a}&;r) \leq
			          \lambda \,\mathfrak E_{p, q} \Bigg\{ \left(1 - \frac{2(\lambda+1)b(c+1)}{c(\lambda+2)(b+1)}\right)
								      \frac{{}_2F_1\left(\eta, \lambda+\eta; \eta+1; -\frac{r^2}{a_1}\right)}{(\lambda+\eta)\,a_1^{\lambda+\eta}}\, 
											\nonumber \\
						&\quad+ \frac{4(\lambda+1) b(c+1)^2}{c(\lambda+2)(b+1)} \frac{a_1^{1-\lambda-\eta}\,
						          {}_2F_1\Big(\eta, \lambda+\eta-1; \eta+1; -\frac{r^2}{a_1}\Big)}
									    {(\lambda+\eta-1)[(\lambda+2)(b+1)r^2+2(c+1)a_1]}\Bigg\} \nonumber \\
						&\quad+ \eta\, \mathfrak E_{p, q} \Bigg\{ \left(1 - \frac{2\lambda b(c+1)}{c(\lambda+1)(b+1)}\right)
								      \frac{{}_2F_1\left(\eta+1, \lambda+\eta; \eta+2; -\frac{r^2}{a_1}\right)}{(\lambda+\eta)\,a_1^{\lambda+\eta}}\, 
											\nonumber \\
						&\quad+ \frac{4\lambda b(c+1)^2}{c(\lambda+1)(b+1)} \frac{a_1^{1-\lambda-\eta}\,
						          {}_2F_1\Big(\eta+1, \lambda+\eta-1; \eta+2; -\frac{r^2}{a_1}\Big)}
									    {(\lambda+\eta-1)[(\lambda+1)(b+1)r^2+2(c+1)a_1]}\Bigg\}\,.
		\end{align}
\end{theorem}

\begin{proof} Firstly consider relation \eqref{C2}:
   \[ \mathfrak{F}_{\lambda,\eta}(F_{p,q};\boldsymbol{a};r) = \lambda \;\mathscr I_{p, q}(\lambda+1,\eta)
                                                          + \eta\; \mathscr I_{p, q}(\lambda,\eta+1)\,,\]
which is found to be bounded above by the auxiliary integral $\mathscr I_{p,q}$  in \eqref{C3}. To do this, we observe that
    \begin{equation}\label{positive}
      F_{p,q}(a, b;c; z) >0 \qquad \left(a \in \mathbb{R}^+,\, c>b>0,\, 0<z<1 \right).
    \end{equation}
 Indeed, it is enough to consider the following known integral
expression \cite[p. 373, Eq. (8.2)]{Ch-Ra-Pa}:
   \[ F_{p,q}(a, b;c; z) = \frac1{{\rm B}(b, c-b)} \int_0^1 \frac{t^{b-1}(1-t)^{c-b-1}}{(1-zt)^a}\,
	                         \exp \left(-\frac pt - \frac q{1-t}\right)\, {\rm d}t\,>0, \]
under the given conditions in \eqref{positive}. Therefore, by virtue of \eqref{G3} and \eqref{G5}, it follows
   \begin{align*}
	    \mathscr I_{p, q}(\lambda, \eta) &= \int_{a_1}^{\infty} \frac{F_{p,q}\,(\lambda,\, b;\,c;\,
			              - \frac{r^{2}}{x})[a^{-1}(x)]}{x^{\lambda}(x+r^{2})^{\eta}}\;{\rm d}x \\
								&\leq \mathfrak E_{p, q} \int_{a_1}^{\infty} \frac{{}_2F_1\,(\lambda,\, b;\,c;\,
			              - \frac{r^{2}}{x})[a^{-1}(x)]}{x^{\lambda}(x+r^{2})^{\eta}}\;{\rm d}x\\
							  &\leq \mathfrak E_{p, q} \Bigg\{ \left(1 - \frac{2\lambda b(c+1)}{c(\lambda+1)(b+1)}\right)
								      \int_{a_1}^{\infty} \frac{[a^{-1}(x)]}{x^{\lambda}(x+r^{2})^{\eta}}\;{\rm d}x \\
								&\qquad + \frac{4\lambda b(c+1)^2}{c(\lambda+1)^2(b+1)^2} \int_{a_1}^{\infty} \frac{[a^{-1}(x)]\;{\rm d}x}
											{x^{\lambda-1}(x+r^{2})^\eta\,\left[r^2+2\frac{(c+1)x}{(\lambda+1)(b+1)}\right]}\Bigg\} \\
								&\leq \mathfrak E_{p, q} \Bigg\{ \left(1-\frac{2\lambda b(c+1)}{c(\lambda+1)(b+1)}\right)
								      \mathscr U_a(\lambda, \eta)  \\
								&\qquad + \frac{4\lambda b(c+1)^2\,\mathscr U_a(\lambda-1, \eta)}
											{c(\lambda+1)(b+1)\,[(\lambda+1)(b+1)r^2+2(c+1)a_1]} \Bigg\}\,.
	 \end{align*}
The rest in deriving \eqref{G6} is obvious.

Secondly, here we recall \eqref{C21} as follows:
  \[ \mathfrak{\widetilde{F}}_{\lambda,\eta}(F_{p,q};\boldsymbol{a};r) = \lambda \;\mathscr{\widetilde{I}}_{p, q}(\lambda+1,\eta)
			                                                           + \eta\; \mathscr{\widetilde{I}}_{p, q}(\lambda,\eta+1)\, ,\]
by  positivity of the integrand of \eqref{C31}, we have
   \[ \mathscr{\widetilde{I}}_{p, q}(\lambda,\eta) \leq \int_{a_1}^{\infty} \frac{F_{p,q}\,(\lambda,\, b;\,c;\,
			           - \frac{r^{2}}{x})}{x^{\lambda}(x+r^{2})^{\eta}}\;{\rm d}x
						\leq \mathfrak E_{p, q} \int_{a_1}^{\infty} \frac{{}_2F_1\,(\lambda,\, b;\,c;\,
			           - \frac{r^{2}}{x})}{x^{\lambda}(x+r^{2})^{\eta}}\;{\rm d}x \, .\]
With the aid of \eqref{G5}, we conclude
   \begin{align*}
	    \mathscr{\widetilde{I}}_{p, q}(\lambda,\eta) &\leq \mathfrak E_{p, q} \Bigg\{ \left(1 -
			                \frac{2\lambda b(c+1)}{c(\lambda+1)(b+1)}\right)
								      \int_{a_1}^{\infty} \frac{{\rm d}x}{x^{\lambda}(x+r^{2})^{\eta}} \\
						&\qquad + \frac{4\lambda b(c+1)^2}{c(\lambda+1)^2(b+1)^2} \int_{a_1}^{\infty} \frac{[a^{-1}(x)]\;{\rm d}x}
											{x^{\lambda-1}(x+r^{2})^\eta\,\left[r^2+2\frac{(c+1)x}{(\lambda+1)(b+1)}\right]}\Bigg\}.
	 \end{align*}
Using \cite[p. 313, Eq. {\bf 3.194} 1.]{GR} for $\lambda+\eta>1$ we have
   \[ \int_{a_1}^{\infty} \frac{{\rm d}x}{x^{\lambda}(x+r^{2})^{\eta}}
	              = \int_0^{\frac1{a_1}} \frac{t^{\lambda+\eta-2}}{(1+r^2t)^{\eta}}\;{\rm d}t
								= \frac{{}_2F_1\left(\eta, \lambda+\eta-1; \eta+1; -\frac{r^2}{a_1}\right)}{(\lambda+\eta-1)\,a_1^{\lambda+\eta-1}}\, , \]
which for $\lambda+\eta>2$  implies
   \begin{align*} 
      \int_{a_1}^{\infty} &\frac{{\rm d}x}{x^{\lambda-1}(x+r^{2})^{\eta}\,[(\lambda+1)(b+1)r^2+2(c+1)x]} \\
	                &\qquad \qquad \qquad \leq \frac{a_1^{2-\lambda-\eta}\, {}_2F_1\Big(\eta, \lambda+\eta-2; \eta+1; -\frac{r^2}{a_1}\Big)}
									     {(\lambda+\eta-2)[(\lambda+1)(b+1)r^2+2(c+1)a_1]} \,.
	 \end{align*}								     
Collecting these formulae we get the upper bound
   \begin{align*}
	    \mathscr{\widetilde{I}}_{p, q}(\lambda,\eta) &\leq \mathfrak E_{p, q} \Bigg\{ \left(1 -
			                \frac{2\lambda b(c+1)}{c(\lambda+1)(b+1)}\right)
								      \frac{{}_2F_1\left(\eta, \lambda+\eta-1; \eta+1; -\frac{r^2}{a_1}\right)}{(\lambda+\eta-1)\,a_1^{\lambda+\eta-1}} \\
						&\qquad + \frac{4\lambda b(c+1)^2}{c(\lambda+1)(b+1)} \frac{a_1^{2-\lambda-\eta}\,
						          {}_2F_1\Big(\eta, \lambda+\eta-2; \eta+1; -\frac{r^2}{a_1}\Big)}
									    {(\lambda+\eta-2)[(\lambda+1)(b+1)r^2+2(c+1)a_1]}\Bigg\} \,.
	 \end{align*}
Now, obvious steps lead to the asserted upper bound \eqref{G7}.
\end{proof}

\vskip 3mm

\begin{remark} \label{rmk-2}
 First observe $\mathfrak E_{p, p} = {\rm e}^{-4p}$. Then the special case of the results in {\rm Theorem \ref{thm-2}}
 can be reduced to yield the simpler upper bound expressions for the respective related Mathieu--type series and its alternating variant
$\mathfrak F_{\lambda,\eta}(F_p;\boldsymbol{a};r)$ and  $\mathfrak{\widetilde{F}}_{\lambda,\eta}(F_p;\boldsymbol{a};r)$.
 Yet their detailed descriptions are left to the interested reader.
\end{remark}

\end{document}